\documentclass[11pt,reqno]{amsart}
\RequirePackage[OT1]{fontenc}
\usepackage{graphicx}
\usepackage{latexsym,color}
\usepackage{amssymb}
\usepackage{amsfonts}
\usepackage{amsmath}
\usepackage{bbm}

\usepackage{float}
\newtheorem{thm}{Theorem}[section]
 \newtheorem{dfn}{Definition}[section]
 \newtheorem{lem}{Lemma}[section]

\newtheorem{prop}{Proposition}[section]

\newenvironment{remark}[1][Remark]{\begin{trivlist}
\item[\hskip \labelsep {\bf Remark}]}{\end{trivlist}}

\numberwithin{equation}{section}

\renewcommand{\thefootnote}{\fnsymbol{footnote}}

\newcommand{\C}{\mathcal{C}}
\newcommand{\E}{\mathbb{E}}
\newcommand{\Pit}{\Pi^{\hat{\tau}}}
\newcommand{\St}{\mathcal{S}^{\hat{\tau}}}

\newcommand{\J}{\mathcal{J}}
\newcommand{\bP}{\mathbb{P}}

\newcommand{\ind}{\mathbbm{1}}

\newcommand{\eps}{\varepsilon}
\newcommand{\cF}{\mathcal{F}^{\hat{\tau}}}
\newcommand{\R}{\mathbb R}

\begin{document}
\title{Disorder Detection with Costly Observations } \thanks{E.B. is supported in part by the National Science Foundation.}\thanks{E.E. is supported by the Knut and Alice Wallenberg Foundation and by the Swedish Research Council.}
\author{Erhan Bayraktar}
\address[Erhan Bayraktar]{Department of Mathematics, University of Michigan, 530 Church Street, Ann Arbor, MI 48104, USA}
\email{erhan@umich.edu}
\author{Erik Ekstr{\"o}m}
\address[Erik Ekstr{\"o}m]{Department of Mathematics, Uppsala University, Box 480, 75106 Uppsala, }
\email{erik.ekstrom@math.uu.se}
\author{Jia Guo}%\footnotemark[4]
\address[Jia Guo]{Department of Mathematics, University of Michigan, 530 Church Street, Ann Arbor, MI 48104, USA}
\email{guojia@umich.edu}

\date{\today}

\maketitle

\renewcommand{\thefootnote}{\arabic{footnote}}

\begin{abstract}
We study the Wiener disorder detection problem where each observation is associated with a positive cost. In this setting, 
a strategy is a pair consisting of a sequence of observation times and a stopping time corresponding to the declaration of 
disorder. We characterize the minimal cost of the disorder problem with costly observations 
as the unique fixed-point of a certain jump operator, and we determine the optimal strategy. 
\end{abstract}

\section{Problem Formulation}

Let $(\Omega, \mathcal{F}, \mathbb{P}_{\pi})$ be a probability space hosting a Brownian motion $W$ and an independent random variable $\Theta$ having distribution
\[
\mathbb{P}_{\pi}\{\Theta=0\}=\pi, \quad \mathbb{P}_{\pi}\{\Theta>t\}=(1-\pi)e^{-\lambda t}, \quad t \geq 0,
\]
where $\pi\in[0,1]$.
We assume that the observation process $(X_t)_{t\geq 0}$ is given by
\begin{equation}
X_t=\alpha(t-\Theta)^++W_t,
\end{equation}
i.e. a Brownian motion which after the random (disorder) time $\Theta$ drifts at rate $\alpha$.
Our objective is to detect the unknown disorder time $\Theta$ based on the observations of $X_t$ as quickly after its occurrence as possible, but at the same time with a small proportion of false alarms. A classical Bayes' risk associated with a stopping strategy $\tau$ (where $\tau$ is a 
stopping time with respect to some appropriate filtration) is given by
\begin{equation}
\label{Bayes}
\mathbb P_\pi(\tau<\Theta) + c\E_\pi[(\tau-\Theta)^+],
\end{equation}
where $c>0$ is a cost associated to the detection delay.

In the classical version of the detection problem, see \cite{S}, observations of the underlying process are costless, and a solution can be obtained by 
making use of the associated formulation in terms of a free-boundary problem. Subsequent literature has, among different things, focused on the case of costly observations. In \cite{B} and \cite{DS}, a version of the problem was 
considered in which observations of increments of the underlying process are costly, and where the cost is proportional 
to the length of the observation time. An alternative set-up was considered in \cite{EBRK}, where the number of observations 
of the underlying process is limited. 

In the current article, we consider a model in which observations of $X$ are unrestricted, but where each observation
is associated with an observation cost $d>0$. We stress the fact that we assume that the controller observes values of 
the process $X$, as opposed to increments of $X$ as in \cite{B} and \cite{DS}. 

Due to the discrete cost of each observation, our observation strategies will consist of finitely many samples; this 
motivates the following definition. 

\begin{dfn}\label{optstr}
A strictly increasing sequence $\hat \tau=\left\{\tau_1,\tau_2,\cdots\right\}$ of random variables is said to belong to $\mathcal T$
if $\tau_1$ is positive and deterministic and if $\tau_j$ is measurable with respect to $\sigma(X_{\tau_1},\cdots,X_{\tau_{j-1}},\tau_1,\cdots,\tau_{j-1})$, $j=2,3,\cdots$. 
For a given sequence $\hat\tau\in\mathcal T$, let 
\[
\mathcal{F}^{\hat{\tau}}_{t}=\sigma(X_{\tau_1},\cdots, X_{\tau_{j}},\tau_1,\cdots,\tau_{j}; \mbox{ where } j=\sup\{k: \tau_k \leq t \}),
\]
let $\mathbb{F}^{\hat{\tau}}=(\mathcal{F}_t^{\hat{\tau}})_{t \geq 0}$, and 
denote by $\St$ the stopping times with respect to $\mathbb{F}^{\hat\tau}$.
\end{dfn}

%Here, $\tau$ represents the last observation taken. Note that taking no observations is an admissible strategy by choosing $\tau=0$. In that case $\mathcal{S}^{\T}$ consists of deterministic times. 

A useful result regarding the structure of the stopping times is the following result which is presented as Proposition~2.1 in \cite{EBRK}.

\begin{lem}\label{lem:st}
Let $\hat\tau\in\mathcal T$, and let $S$ be an $\mathbb{F}^{\hat\tau}$-stopping time. Then for each $j\geq 1$, both  $S 1_{\{\tau_j \leq S <\tau_{j+1}\}}$ and $1_{\{\tau_j \leq S <\tau_{j+1}\}}$ are $\mathbb{F}^{\hat\tau}_{\tau_j}$-measurable.
\end{lem}

We generalize the Bayes' risk defined in \eqref{Bayes} by formulating the quickest detection problem with observation costs as 
\begin{equation}\label{minimalrisk}
\begin{split}
V(\pi)&=\inf_{\hat{\tau}\in\mathcal{T}}\inf_{\tau\in \St } \left\{\mathbb P_{\pi}({\tau}<\Theta)+\mathbb E_{\pi}\left[c\, ({\tau}-\Theta)^++d\sum_{k=1}^{\infty}\mathbbm{1}_{\{\tau_k\leq \tau\}}\right]\right\}.
\end{split}
\end{equation}
Here the positive constant $c$ represents the cost of detection delay, and the positive constant $d$ represents the cost for 
each observation. 
Note that the observer has two controls: she controls the observation sequence $\hat{\tau}$, and also needs to decide when the change happened, which is the role of $\tau$. 
We should point out that we can extend our results to cover ``expected miss" criterion using the results of \cite{MR2065992}, by appropriately choosing the cost parameters. We can also consider the exponential delay penalty as in \cite{MR3077545} but in this case one needs to use additional sufficient statistics need to be used. This would change the nature of the problem, but we expect the structure of the solution to be qualitatively similar. For further discussion on different criteria and the derivation of sufficient statistics also see \cite{MR2158013}.

Problem~\eqref{minimalrisk} can be formulated as a control problem in terms of the a posteriori probability process
\begin{equation}
\Pit_t:=\mathbb P_{\pi}(\Theta\leq t\big |\cF_t)%, \quad \Pi^{\hat{\tau}}_0=\pi
\end{equation}
as
\begin{equation}\label{eq:valuefunc}
V(\pi)=\inf_{\hat{\tau}\in\mathcal{T}}\inf_{{\tau}\in\St}\rho^{\pi}(\hat{\tau},\tau),
\end{equation}
where
\[
\rho^{\pi}(\hat{\tau},\tau):=\E_\pi\left[1-\Pit_{\tau}+c \int_0^{\tau}\Pit_s ds+d\sum_{k=1}^{\infty}\mathbbm{1}_{\{\tau_k\leq \tau\}}\right]. 
\]
The computations are analogous to, e.g., \cite[Proposition 5.8]{MR2482527}.  Observe that we can restrict ourselves to stopping times with $\mathbb E[\tau]<\infty$.

\begin{remark}
Clearly, $V(\pi)\geq 0$. Moreover, choosing $\tau=0$ yields $V(\pi) \leq 1-\pi$.
\end{remark}

For $\pi=1$, the a posteriori probability process $\Pi^{\hat{\tau}}_{t}$ is constantly equal to 1.
If $\pi\in[0,1)$, then $\Pi^{\hat{\tau}}_{t}$ can (see \cite{EBRK} and \cite{D}) be expressed recursively as
\begin{equation}\label{eq:defn-pi}
 \Pit_t= \left\{
\begin{array}{ll}
\pi & t=0,\\
1-(1-\Pi^{\hat{\tau}}_{\tau_{k-1}})e^{-\lambda(t-\tau_{k-1})} & \tau_{k-1} < t< \tau_{k}, \\
\frac{j(\tau_{k}-\tau_{k-1},\Pi^{\hat{\tau}}_{\tau_{k-1}},X_{\tau_{k}}-X_{\tau_{k-1}})}{1+j(\tau_{k}-\tau_{k-1},\Pi^{\hat{\tau}}_{\tau_{k-1}},X_{\tau_k}-X_{\tau_{k-1}})} & t=\tau_{k},    
     \end{array} 
\right. 
\end{equation}
where $k\geq 1$, $\tau_0:=0$, and
\begin{align*}
j(t,\pi,x)=
\exp\left\{\alpha x+(\lambda-\frac{\alpha^2}{2})t  \right\}\frac{\pi}{1-\pi}+\lambda\int_0^t\exp\left\{(\lambda+\frac{\alpha x}{t})u-\frac{\alpha^2u^2}{2t}  \right\}du.
\end{align*}
Thus at an observation time $\tau_k$, the process $\Pit$ jumps from 
\[1-(1-\Pi^{\hat{\tau}}_{\tau_{k-1}})e^{-\lambda(\tau_k-\tau_{k-1})}\] 
to 
\[ \frac{j(\tau_{k}-\tau_{k-1},\Pi^{\hat{\tau}}_{\tau_{k-1}},X_{\tau_{k}}-X_{\tau_{k-1}})}{1+j(\tau_{k}-\tau_{k-1},\Pi^{\hat{\tau}}_{\tau_{k-1}},X_{\tau_k}-X_{\tau_{k-1}})}.\]
Moreover, $(t, \Pit_t)$ with respect to $\mathbb{F}^{\hat\tau}$ is a piece-wise deterministic Markov process in the sense of \cite[Section 2]{MR1283589} and therefore has the strong Markov property. 

At time $t=0$, the observer could decide that he will not be making any observations (by setting $\tau_1=\infty$). Then $\Pit$ evolves deterministically (see \eqref{eq:defn-pi}), and the corresponding cost of following that strategy is thus given by
\begin{eqnarray*}
F(\pi) &=& \inf_{t\geq 0}\left\{1-\Pit_{t}+c \int_0^{t}\Pit_s ds\right\} \\
&=& \inf_{t\geq 0}\left\{(1-\pi)e^{-\lambda t}+ct-c(1-\pi)\frac{1-e^{-\lambda t}}{\lambda}  \right\} \\
&=& \left\{\begin{array}{ll}
 \frac{c}{\lambda} \left(\pi+\log\frac{(\lambda+c)(1-\pi)}{c}\right) & \pi < \frac{\lambda}{c+\lambda};\\
 1-\pi &  \pi  \geq \frac{\lambda}{c+\lambda}.
\end{array}\right.
\end{eqnarray*}
Moreover, the optimizer $t^*$ is given by
\begin{equation}
t^*(\pi)=\left\{\begin{array}{ll}
\frac{1}{\lambda}\log\frac{(\lambda+c)(1-\pi)}{c} & \pi < \frac{\lambda}{c+\lambda};\\
0 &  \pi  \geq \frac{\lambda}{c+\lambda}.\end{array}\right.
\end{equation}

For a given sequence $\hat\tau\in\mathcal T$ of observations, let $\mathcal S^{\hat\tau}_0\subseteq \mathcal S^{\hat\tau}$ denote the set of $\mathbb F^{\hat\tau}$-stopping times $\tau$ such that $\mathbb P_\pi$-a.s. $\tau=\tau_k$ for some $k=k(\omega)$.

\begin{prop}\label{prop1.1}
The quickest detection problem with costly observations $V(\pi)$ in (\ref{minimalrisk}) can be represented as
\begin{equation}
\begin{split}
V(\pi)&=\inf_{\hat{\tau}\in\mathcal{T}}\inf_{\tau\in \mathcal S^{\hat\tau}_0}
\mathbb E_{\pi}\Big[F(\Pi^{\hat{\tau}}_{\tau})+c\tau-\frac{c}{\lambda}\sum_{k=0}^{\infty}(1-\Pi^{\hat{\tau}}_{\tau_k})(1-e^{-\lambda(\tau_{k+1}-\tau_k)})\mathbbm{1}_{\{\tau_{k+1}\leq \tau\}}\\
&\hspace{45mm}+d\sum_{k=1}^{\infty}\mathbbm{1}_{\{\tau_k\leq \tau\}}\Big],
\end{split}
\end{equation}
i.e. the value function is a combined optimal stopping and impulse control problem.
\end{prop}

\begin{proof}
It follows from Lemma~\ref{lem:st} that any stopping time $\bar{\tau}\in \St$ can be written as $\bar{\tau}=\tau+\bar{t}$, for $\tau\in \mathcal S^{\hat\tau}_0$ and for some $\mathbb F^{\hat\tau}_\tau$-measurable random variable 
$\bar{t}$. Then by conditioning at $\tau$ first, optimizing over the stopping times in $\St$ and then taking expectations we obtain
\begin{equation}
\label{eq:valuefunc0}
V(\pi)=\inf_{\hat{\tau}\in\mathcal{T}}\inf_{\tau\in \mathcal S^{\hat\tau}_0}\E_\pi\left[F(\Pi_{\tau})+c \int_0^{\tau}\Pit_s ds+d\sum_{k=1}^{\infty}\mathbbm{1}_{\{\tau_k\leq \tau\}}\right].
\end{equation}
The rest of the proof can be done using \eqref{eq:defn-pi} and partitioning the integral into integrals over $[\tau_i,\tau_{i+1})$.
\end{proof}

%
%\begin{rem}
%It is clear from the above proof that given the last observation time $\tau$, the optimal time to stop, is given by $\btau^*=\tau+\bt^*(\Pit_{\tau})$. \end{rem}

In a related work \cite{DE}, the sequential hypothesis testing problem for the drift of a Wiener process was considered under the same assumption of costly observations. In the sequential hypothesis testing problem, the nature of the data remains the same and the goal is to determine the nature of data as soon as possible. In contrast in the quickest detection problem, the nature of data changes at a given time and the goal is to determine the change time as soon as it happens by balancing detection delay and false alarm frequency. As a result of this difference, the evolution of the sufficient statistic $\Pi$ and the pay-off function are different. This for example leads to a different functional operator in the next section. Although both papers use monotone functional operators that preserve concavity, because of the form of our operator the proof of preservation of concavity  is far more difficult in our case and requires a new idea.  We will also see that the solution structure is different: For example in the Quickest detection problem although there is no observation rights left the decision maker still does not declare the decision and waits more; see Section~\ref{sec:opt-strategy}.

\section{A functional characterization of the value function}

In this section we study the value function $V$ and show that it is a fixed-point of a monotone operator $\J$, in effect proving the so-called dynamic programming principle. This will then be used in the next section to construct an optimal strategy. 

To define the operator $\J$, let 
\[\mathbb F:=\{f:[0,1]\to[0,1]\mbox{ measurable and such that }f(\pi)\leq 1-\pi\}\]
and set 
$$
\mathcal{J}f(\pi)=\min\{F(\pi),\inf_{t>0}\mathcal{J}_0f(\pi,t)\}
$$
for $f\in\mathbb F$, where
\[\J_0f(\pi,t)=d+\E_\pi\left[f\left(\frac{j(t,\pi,X_t)}{1+j(t,\pi,X_t)}\right) + c(t-\Theta)^+\right].\]
Note that 
\[\E_\pi\left[(t-\Theta)^+\right]= t-(1-\pi)\frac{1-e^{-\lambda t}}{\lambda},\]
so
\begin{equation}
\label{J0}
\J_0f(\pi,t)=d+\E_\pi\left[f\left(\frac{j(t,\pi,X_t)}{1+j(t,\pi,X_t)}\right) + c t-c(1-\pi)\frac{1-e^{-\lambda t}}{\lambda}  \right].
\end{equation}

\begin{prop}\label{properties}
The operator $\J$
\begin{itemize}
\item[(i)]
is monotone: $f_1\leq f_2\implies \J f_1\leq \J f_2$;
\item[(ii)]
is concave: $\mathcal J(af_1+(1-a)f_2)\geq a\J f_1+(1-a)\J f_2$ for $a\in[0,1]$;
\item[(iii)]
satisfies $\J 0(\pi)=\min \{F(\pi),d\}$;
\item[(iv)]
has at most one fixed point $f\in\mathbb F$ such that $f=\J f$;
\item[(v)]
is concavity preserving: if $f\in\mathbb F$ is concave, then also $\J f$ is concave.
\end{itemize}
\end{prop}

\begin{proof}
(i) and (iii) are immediate. For (ii), let $f_1,f_2\in\mathbb F$ and let $a\in[0,1]$. Then
\begin{eqnarray*}
\J(af_1+(1-a)f_2) &=& \min\left\{F, \inf_{t>0} \left\{ a\J_0f_1 +(1-a)\J_0f_2\right\}\right\}\\
&\geq& \inf_{t>0}\left\{ a\min\{F,\J_0 f_1\} + (1-a)\min\{F,\J_0 f_2\}\right\}\\
&\geq& a \J f_1 +(1-a)\J f_2.
\end{eqnarray*}

For (iv) we argue as in \cite[Lemma 54.21]{MR1283589};
assume that there exist two distinct fixed points of $\J$, i.e. $f_1=\J f_1$ and $f_2=\J f_2$ for $f_1,f_2\in\mathbb F$ such that
$f_1(\pi)<f_2(\pi)$ (without loss of generality) for some $\pi\in[0,1)$. 
Let $a_0:=\sup\{a\in[0,1]: af_2\leq f_1\}$, and note that $a_0\in[0,1)$. 
From (iii) it follows that there exists $\kappa>0$ such that $\kappa \J 0\geq 1-\pi$, $\pi\in[0,1]$, so using 
(i) and (ii) we get
\[f_1 = \J f_1 \geq \J(a_0f_2) \geq a_0\J f_2 + (1-a_0)\J 0\\
\geq (a_0+(1-a_0)\kappa^{-1})f_2,\]
which contradicts the definition of $a_0$.

For (v), first note that $F$ is concave. Since the infimum of concave functions is again concave, it therefore follows from \eqref{J0} that it suffices to check that 
\[
\E_\pi\left[f\left(\frac{j(t,\pi,X_t)}{1+j(t,\pi,X_t)}\right)\right]\]
is concave in $\pi$ for any $t>0$ given and fixed. To do that, define measures $\mathbb Q_{\pi}$, $\pi\in[0,1)$, on $\sigma(X_t)$ by
\[d\mathbb Q_\pi := \frac{e^{\lambda t}}{(1-\pi)(1+j(t,\pi,X_t))} d\mathbb P_\pi.\]
Then 
\[\E_\pi\left[\frac{d\mathbb Q_\pi}{d\mathbb P_\pi}\right]=\frac{e^{\lambda t}}{1-\pi}\int_{\R} \frac{1}{1+j(t,\pi,y)}\mathbb P_{\pi}(X_t\in dy).\]
Denoting by $\varphi$ the density of the standard normal distribution, we have
\begin{eqnarray*}
 \frac{ \mathbb P_{\pi}(X_t\in dy) }{1-\pi}
&=&  \frac{\pi}{1-\pi} \bP_\pi(X_t \in dy\vert \Theta=0) + \lambda\int_0^{t}  \bP_\pi(X_t \in dy\vert \Theta=s)e^{-\lambda s}ds\\
&&+ \bP_{\pi}(X_t \in dy\vert \Theta>t)e^{-\lambda t} \\
&=& \frac{\pi}{(1-\pi)\sqrt t} \varphi\left(\frac{y-\alpha t}{\sqrt t}\right) 
+ \frac{\lambda}{\sqrt t}\int_0^te^{-\lambda s}\varphi\left(\frac{y-\alpha (t-s)}{\sqrt t}\right)ds\\
&& + \frac{e^{-\lambda t}}{\sqrt t}\varphi\left(\frac{y}{\sqrt t}\right)\\
&=& e^{-\lambda t}(1+j(t,\pi,y))\varphi\left(\frac{y}{\sqrt t}\right). 
\end{eqnarray*}
Thus 
\[\E_\pi\left[\frac{d\mathbb Q_\pi}{d\mathbb P_\pi}\right]=\frac{1}{\sqrt t}\int_\R\varphi\left(\frac{y}{\sqrt t}\right)\,dy =1\]
so $\mathbb Q_\pi$ is a probability measure. Furthermore, the random variable $X_t$ is 
$N(0,t)$-distributed under $\mathbb Q_\pi$;
in particular, the $\mathbb Q_\pi$-distribution of $X_t$ does not depend on $\pi$.

Since $j(t,\pi,x)$ is affine in $\pi/(1-\pi)$, the function 
\[\pi\mapsto (1-\pi)f\left(\frac{j(t,\pi,x)}{1+j(t,\pi,x)}\right)(1+j(t,\pi,x))\] 
is concave if $f$ is concave. It thus follows from
\begin{eqnarray*}
&& \hspace{-10mm} \E_\pi\left[f\left(\frac{j(t,\pi,X_t)}{1+j(t,\pi,X_t)}\right)\right] \\
&=&
(1-\pi)\exp\{-\lambda t\} \E^{\mathbb Q_\pi}\left[f\left(\frac{j(t,\pi,X_t)}{1+j(t,\pi,X_t)}\right)(1+j(t,\pi,X_t))\right]
\end{eqnarray*}
and \eqref{J0} that $\pi\mapsto \mathcal J_0f(\pi,t)$ is concave, which completes the proof.
\end{proof}

Next we define a sequence $\{f_n\}_{n=0}^{\infty}$ of functions on $[0,1]$ by setting
$$
f_0(\pi)=F(\pi), \quad f_{n+1}(\pi)=\mathcal{J}f_n(\pi), \quad  n\geq 0.
$$

\begin{prop}\label{prop:decr-con}
For $\{f_n\}_{n=1}^\infty$ we have that
\begin{itemize}
\item[(i)]
the sequence is decreasing;
\item[(ii)]
each $f_n$ is concave.
\end{itemize}
\end{prop}

\begin{proof}
Clearly, $f_1 \leq F=f_0$, so Proposition~\ref{properties} (i) and a straightforward induction argument give that $f_n$ is  decreasing in $n$. Hence the pointwise limit $f_{\infty}:=\lim_{n\to\infty}f_n$ exists.
Furthermore, 
since $F$ is concave, each $f_n$ is concave by Proposition~\ref{properties} (v). 
\end{proof}

Thus the pointwise limit $f_{\infty}:=\lim_{n\to\infty}f_n$ exists. 
Since the pointwise limit of concave functions is concave, it follows that also $f_\infty$ is concave.

\begin{prop} \label{measurable}
Let $f\in\mathbb F$ be continuous.
For fixed $\pi\in[0,1]$, the function $t \mapsto \mathcal{J}_0f(\pi,t)$ attains its minimum for some point $t \in [0,\infty)$. Denote the first of these minimums by 
$t(\pi,f)$, i.e.
\begin{equation}
\label{t}
t(\pi,f):= \inf \{t \geq 0:  \inf_s \mathcal{J}_0f(\pi,s) =  \mathcal{J}_0f(\pi,t) \}.
\end{equation}
Then $\pi \mapsto t(\pi,f)$ is measurable. 
\end{prop}

\begin{proof}
Observe that $(t,\pi) \mapsto \mathcal{J}_0f(\pi,t)$ is a finite continuous function which 
approaches $\infty$ as $t \to \infty$. It follows that  $t (\pi,f)$ is finite.

We will prove the measurability of $\pi \mapsto t(\pi,f)$ by showing that it is lower semi-continuous. 
Let $\pi_i \to \pi_{\infty}$ and let $t_i=t(\pi_i,f)$. Because $t \to ct$ is the dominating term in $t \mapsto \J_0 f(\pi,t)$, 
it is clear that the sequence $\{t_i\}_{i \in \mathbb{N}}$ is bounded. 
It follows that $t_{\infty}:=\liminf t_i<\infty$; let $\{t_{i_j}\}_{j=1}^\infty$ be a subsequence such that 
$t_{i_j}\to t_\infty$. Then, by the Fatou lemma, 
\[
\J_{0}f(\pi_{\infty}, t_{\infty}) \leq \liminf_{j \to \infty} \J_0f(\pi_{i_j}, t_{i_j}) = 
\lim_{j \to \infty} \J_0 f(\pi_{i_j},t_{i_j}) = \J_{0} f(\pi_{\infty},t_\infty).
\]
Thus 
\[
t(\pi_{\infty},f) \leq t_{\infty} = \liminf_{i\to\infty}t (\pi_i,f),
\]
which establishes the desired lower semi-continuity. 
\end{proof}

\begin{prop}
The function $f_\infty$ is the unique fixed point of the operator $\J$.
\end{prop}

\begin{proof}
Since the operator $\J$ is monotone and $f_{n} \geq f_{\infty}$, it is clear that $f_{\infty} \geq \J f_{\infty}$. On the other hand,
\[
f_{n+1}(\pi)= \J f_n(\pi) \leq \min\{F(\pi), \J_0 f_n(\pi, t(\pi,f_\infty))\},
\]
where $t(\pi,f_\infty)$ is defined as in \eqref{t}. Letting $n \to \infty$ and using the monotone convergence theorem we obtain that $f_{\infty}$ is a fixed point. Since uniqueness is established in Proposition~\ref{properties}, this completes the proof.
\end{proof}

Next we introduce the problem of an agent who is allowed to make at most $n$ observations:
\begin{equation}
V_n(\pi):=\inf_{\hat{\tau}\in\mathcal{T}}\inf_{\tau\in \mathcal S_0^{\hat\tau},\tau\leq \tau_n}\rho^{\pi}(\hat{\tau},\tau).
\end{equation}
These functions can be sequentially generated using the integral operator $\J$.

\begin{prop}
We have $V_n=f_n$, $n\geq 0$.
\end{prop}
\begin{proof}
First note that $V_0=f_0=F$. Now assume that $V_{n-1}=f_{n-1}$ for some $n\geq 1$.

\noindent \textbf{Step 1: $V_n(\pi)\geq f_n(\pi)$.} 

For any $\hat{\tau}\in\mathcal{T}$ and $\tau\in \mathcal S_0^{\hat\tau}$ we have
\begin{eqnarray}\label{eq:esdi}
&&\hspace{-10mm}\mathbb E_{\pi}\left[F(\Pi_{\tau\wedge\tau_n})+c \int_0^{\tau \wedge \tau_n }\Pit_s ds
+d\sum_{k=1}^{\infty}\mathbbm{1}_{\{\tau_k\leq \tau\wedge\tau_n\}}\right]\\
\notag
&=&\mathbb E_{\pi}\left[\mathbbm{1}_{\{\tau_1=0\}}F(\pi)\right]\\
&&\notag
 +\mathbb E_{\pi}\left[\mathbbm{1}_{\{\tau_1>0\}}\left(F(\Pi_{\tau\wedge\tau_n})+c \int_0^{\tau \wedge \tau_n }\Pit_s ds
+d\sum_{k=1}^{\infty}\mathbbm{1}_{\{\tau_k\leq \tau\wedge\tau_n\}}\right)\right] \\
\notag
&\geq&\mathbbm{1}_{\{\tau_1=0\}}F(\pi)
+\mathbbm{1}_{\{\tau_1>0\}} \mathbb E_{\pi}\left[ \left(d+c \int_{0}^{\tau_1} \Pit_s ds +V_{n-1}(\Pi_{\tau_1})\right)\right]\\
\notag
&=&\mathbbm{1}_{\{\tau_1=0\}}F(\pi)
+ \mathbbm{1}_{\{\tau_1>0\}}\mathbb E_{\pi}\left[\left(d+c \int_{0}^{\tau_1} \Pit_s ds
+f_{n-1}(\Pi_{\tau_1})\right)\right]\\
\notag
&\geq & \J f_{n-1}(\pi)=f_n(\pi),
\end{eqnarray}
where we used the fact that $\tau_1$ is deterministic and the Markov property of $\Pit$. We obtain the desired result from \eqref{eq:esdi} by taking the infimum over strategy pairs $(\hat\tau,\tau)$.

\textbf{Step 2: $V_n(\pi)\leq f_n(\pi)$.} 

We only need to prove this for the case $\J f_{n-1}(\pi) < F(\pi)$ (since otherwise $f_{n}(\pi) = \J f_{n-1}(\pi)=F(\pi) \geq V_n(\pi)$ already).

Note that $V_{0}=F=f_0$. We will assume that the assertion holds for $n-1$ and then prove it for $n$. We will follow ideas used in the proof of Theorem 4.1 in \cite{MR2806570}. Denoting $t_{n}:=t(\pi, f_{n-1})$, let us introduce a sequence $\hat\tau$
of stopping times
\begin{equation}\label{eq:eps-opt}
\tau_1=t_n, \quad \tau_{i+1}= \sum_{k}\tau^k_i \circ \theta_{t_n} \ind_{\{\Pit_{t_n} \in B_k\}}, \quad i=1,\cdots,n-1,
\end{equation}
where $(B_k)_k$ is a finite partition of $[0,1)$ by intervals and $\tau^k$ are $\epsilon$-optimal observation times for when the process $\Pi$ starts from the centre of these intervals. 
\footnote{$\theta$ is the shift operator in the Markov process theory, see e.g. \cite{MR2767184}}

Since $V_{n-1}$ is continuous, and the expected value (before optimizing) is a continuous function of the initial starting point for any strategy choice, which is due to the continuity of $\Pi$ with respect to its starting point, the above sequence is a $O(\eps)$ if the intervals are chosen to be fine enough.

Now we can write
\begin{eqnarray*}
 f_n(\pi) 
&=& ct_n +d-\frac{c}{\lambda} (1-\pi) (1-e^{-\lambda t_n})+\mathbb E_{\pi}[V_{n-1}(\Pi^{\hat{\tau}}_{t_n})]\\
&\geq& ct_n+d-\frac{c}{\lambda}(1-\pi)(1-e^{-\lambda t_n})-O(\epsilon)\\
&&\hspace{-10mm}+\mathbb E_{\pi}\left[\mathbb E_{\pi}\left[\left(F(\Pit_{\tau\wedge\tau_{n-1}})  + \int_0^{\tau_{n-1} \wedge \tau} \Pit_s ds +d\sum_{k=1}^{\infty}\mathbbm{1}_{\{\tau_k\leq \tau\wedge\tau_{n-1}\}}\right) \circ \theta_{t_n}\bigg| \cF_{t_n}\right] \right]\\
&=&\mathbb E_{\pi}\Big[F(\Pit_{\tau\wedge\tau_n}) + \int_0^{\tau \wedge \tau_n}\Pit_s ds +d\sum_{k=1}^{\infty}\mathbbm{1}_{\{\tau_k\leq \tau\wedge\tau_n\}}\Big] -O(\epsilon)\\
&\geq& V_n(\pi)-O(\epsilon),
\end{eqnarray*}
where we used the fact that
\[c\int_0^{t_n}\Pit_s ds = ct_n -\frac{c}{\lambda} (1-\pi) (1-e^{-\lambda t_n}).\]
Since $\epsilon>0$ can be made arbitrary small, this shows that $V_n(\pi)\leq f_n(\pi)$.
\end{proof}

\begin{prop}
We have that $V=f_{\infty}$, i.e., $V$ is the unique fixed point of $\J$.
\end{prop}

\begin{proof}
Since $V_n=f_n\to f_\infty$, it suffices to show $\lim_{n\to\infty}V_n=V$. It follows by definition that $V(\pi)\leq V_n(\pi)$ for any $n\geq 1$ and $\pi\in[0,1]$. We thus only need to prove that $\lim_{n}V_n(\pi) \leq V(\pi)$. Assume that a pair $(\hat\tau, \tau)$ where $\hat{\tau}\in\mathcal{T}$ 
and $\tau\in\St_0$ is an $\epsilon$-optimizer for \eqref{eq:valuefunc0}.
Then
\begin{eqnarray}\label{eq:anstpen}
V_n(\pi) &\leq& \mathbb E\left[F(\Pit_{\tau\wedge\tau_n})+\int_0^{\tau \wedge \tau_n}\Pit_s ds+d\sum_{k=1}^{\infty}\mathbbm{1}_{\{\tau_k\leq \tau\wedge\tau_n\}}\right]\\
\notag
&\leq& \mathbb E\left[F(\Pit_{\tau\wedge\tau_n})+\int_0^{\tau}\Pit_s ds+d\sum_{k=1}^{\infty}\mathbbm{1}_{\{\tau_k\leq \tau\}}\right].
\end{eqnarray}
Note that since $\tau(\omega)= \tau_k(\omega)$ for some $k=k(\omega)$, we have  
$\Pit_{\tau \wedge \tau_n}(w)= \Pit_{\tau}(\omega)$ if $n\geq k(\omega)$. 
As a result, and since $F$ is bounded and continuous, the bounded convergence theorem applied to 
\eqref{eq:anstpen} gives 
\[\lim_{n\to\infty} V_n(\pi)\leq V(\pi)+\epsilon.\]
Since $\epsilon>0$ is arbitrary, this completes the proof.
\end{proof}

This approach to proving the DPP in the context of Quickest Detection problems goes back to \cite{MR2260062}. There, the observations were coming from a process with jumps and the operator was defined through jump times of the observation process. On the other hand, the operator is defined through the observation times (which are now also part of decision makers choice set).

\section{The optimal strategy}\label{sec:opt-strategy}

In this section we study the optimal strategy for the detection problem with costly observations. 
More precisely, we seek to determine an optimal distribution of observation times $\hat\tau$ and an optimal stopping time 
$\tau$. The optimal strategy is determined in terms of the continuation region
\[
\C:=\{\pi \in [0,1]: V(\pi)< F(\pi)\}.
\]
Note that for $\pi\in \C$ we have
$$
V(\pi)=\inf_{t\ge 0}\mathcal{J}_0 V(\pi,t)
$$
thanks to our main result from the last section. Denote by $t(\pi):=t(\pi,f_\infty)=t(\pi,V)$, and note that since $\mathcal J_0V(\pi,0)=d+V(\pi)$, 
we have $t(\pi)>0$ on $\C$.

Moreover, define $t^*$ by
\[t^*(\pi)= \left\{\begin{array}{ll} t(\pi) &\mbox{for }\pi\in \C\\
\infty& \mbox{for }\pi\notin \C\end{array}\right.\]
Using the function $t^*$, we construct recursively an observation sequence $\hat\tau^*$ 
and a stopping time $\tau^*$ as follows.

Denote by $\tau_0^*=0$ and $\Pi_0=\pi$. For $k= 1,2...$, define recursively
\[\tau^*_{k}:=\tau^*_{k-1}+t^*(\Pi_{\tau^*_{k-1}})\] 
and 
\[\Pi_{\tau^*_k} :=
\frac{j(\tau^*_{k}-\tau^*_{k-1},\Pi_{\tau^*_{k-1}},X_{\tau^*_{k}}-X_{\tau^*_{k-1}})}{1+j(\tau^*_{k}-\tau^*_{k-1},\Pi_{\tau^*_{k-1}},X_{\tau^*_k}-X_{\tau^*_{k-1}})}.\]
Then $\hat\tau^*:=\{\tau_k^*\}_{k=1}^{\infty}\in\mathcal{T}$. 
Moreover, let 
\[n^*:= \min\{k\geq 0: \Pi_{\tau^*_{k}}\notin \C\}= \min\{k\geq 0:\tau^*_{k}=\infty\},\]
and define $\tau^*:=\tau^*_{n^*}$. Then $\tau^*\in\mathcal S^{\hat\tau^*}$, and $n^*$
is the total number of finite observation times in $\hat\tau^*$.

\begin{thm}
The strategy pair $(\hat\tau^*,\tau^*)$ is an optimal strategy.
\end{thm}

\begin{proof}
Denote by
\begin{equation*}
V^*(\pi)=
\mathbb E_{\pi}\Big[F(\Pi_{\tau^*})+c\tau^*-\frac{c}{\lambda}\sum_{k=0}^{n^*-1}(1-\Pi_{\tau^*_k})(1-e^{-\lambda(\tau^*_{k+1}-\tau^*_k)})+dn^*\Big],
\end{equation*}
Clearly, by the definition of $V$, we have $V^*(\pi)\ge V(\pi)$. It thus remains to show $V\ge V^*(\pi)$. 

For $n\geq 0$, let $\tau_n^\prime := \tau^*_n\wedge\tau^*=\tau^*_{n\wedge n^*}$.

\noindent
{\bf Claim:} We have 
\begin{eqnarray}
\label{claim}
\notag
V(\pi) &=& \mathbb E_{\pi}\left[V(\Pi_{\tau^\prime_n})+c\tau^\prime_{n}-\frac{c}{\lambda}\sum_{k=0}^{n\wedge n^*-1}(1-\Pi_{\tau^*_k})(1-e^{-\lambda(\tau^*_{k+1}-\tau^*_k)})\right]\\
&&
+\E_\pi\left[d(n\wedge n^*)\right]\\
\notag
&=:& RHS(n) 
\end{eqnarray}
for all $n\geq 0$.

To prove the claim, first note that $\tau_0^\prime=0$, so $V(\pi)=RHS(0)$. 
Furthermore, by the Markov property we have
\begin{eqnarray*}
&& \hspace{-15mm} RHS(n+1) -RHS(n) \\
&=&
\mathbb E_{\pi}\left[\left(V(\Pi_{\tau^*_{n+1}})-V(\Pi_{\tau^*_n}) +c(\tau^*_{n+1}-\tau^*_n)\right.\right.\\
&&\hspace{10mm} -\left.\left.
\frac{c}{\lambda}(1-\Pi_{\tau^*_n})(1-e^{-\lambda(\tau^*_{n+1}-\tau^*_n)})+d\right)\mathbbm{1}_{\{n^*\ge n+1\}}\right]\\
&=& \mathbb E_{\pi}\left[\left(\mathbb E_{\Pi_{\tau_n^*}}\left[V(\Pi_{\tau_1^*}) + c\tau_1^* \right]- V(\Pi_{\tau^*_n})\right.\right.\\
&&\hspace{10mm} - \left.\left. \frac{c}{\lambda}(1-\Pi_{\tau_n^*}) \mathbb E_{\Pi_{\tau_n^*}}\left[ 1-e^{-\lambda \tau_1^*} \right] 
+d  \right)\mathbbm{1}_{\{n^*> n\}}\right]
\\
&=& 0,
\end{eqnarray*}
which shows that \eqref{claim} holds for all $n\geq 0$.

Note that it follows from \eqref{claim} that $n^*<\infty$ a.s. (since otherwise the term $\E_\pi[d(n\wedge n^*)]$ would explode as $n\to\infty$). Therefore, letting $n\to\infty$ in \eqref{claim}, using bounded convergence and monotone convergence, we find that
\begin{eqnarray*}
V(\pi) &=& \mathbb E_{\pi}\left[V(\Pi_{\tau^*})+c\tau^*-\frac{c}{\lambda}\sum_{k=0}^{ n^*-1}(1-\Pi_{\tau^*_k})(1-e^{-\lambda(\tau^*_{k+1}-\tau^*_k)}) + d n^*\right]\\
&=& \mathbb E_{\pi}\left[F(\Pi_{\tau^*})+c\tau^*-\frac{c}{\lambda}\sum_{k=0}^{ n^*-1}(1-\Pi_{\tau^*_k})(1-e^{-\lambda(\tau^*_{k+1}-\tau^*_k)}) + d n^*\right]\\
&=& V^*(\pi),
\end{eqnarray*}
which completes the proof.
\end{proof}

Our approach, which relies on dynamic programming principle, should be contrasted with the verification approach used in \cite{MR2374974, MR2256030, MR2482527} in which one first finds a smooth enough solution to a free boundary problem and uses It\^{o}'s formula to verify that this solution is the value function. Another useful outcome of our approach due to its iterative nature is its usefulness for a numerical approximation.

\section{Numerical Examples}

In Figure \ref{41new}, we illustrate Proposition~\ref{prop:decr-con}. We use the same parameters that were used for 
Figure~2 in \cite{EBRK}, where $d=0$ .

Clearly, the value functions $V_n$ increase in the cost parameters. Figure~\ref{42new} displays the value functions $V_1,...,V_{10}$ for the same parameters as in Figure~\ref{41new} but for a larger cost $c$. Similarly,
the sensitivity with respect to the observation cost parameter $d$ is pictured in Figure~\ref{43new}. 
x
In Figure~\ref{44new.png} we compute the function $t$ defined in \eqref{t}, when $f$ in the definition is replaced by $V_n$, for various values of $n$. 
While it appears that $t(\pi,V_n)$ is decreasing in $n$ (the more observation rights one has, the more inclined one is
to make early observations) and decreasing in $\pi$, we have not been able to prove these monotonicities.

Finally, in Figure~\ref{44new1.png} we determine $\pi^*(n)=\inf\{\pi: t^*(\pi,V_n)=\infty\}$. Our observations consistently indicate that the continuation region for taking observations is an interval of the form $[0, \pi^*(n))$; also here, an analytical proof of this remains to be found.

\begin{figure}[H]
\begin{center}
\includegraphics[scale=.5]{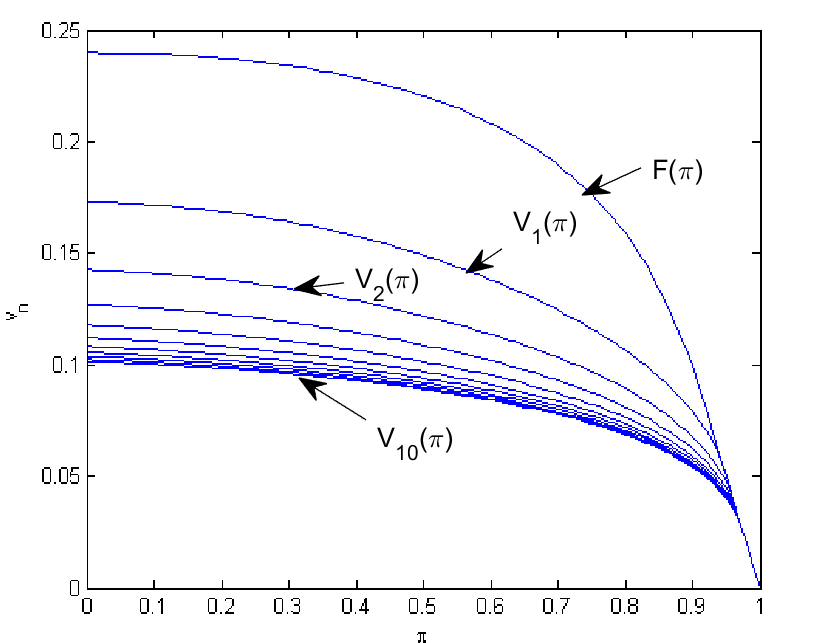}
\end{center}
\caption{$c=0.01,\lambda=0.1,\alpha=1,d=0.001,n=0,1,\cdots,10$.}\label{41new}
\end{figure}
\begin{figure}[H]
\begin{center}
\includegraphics[scale=.5]{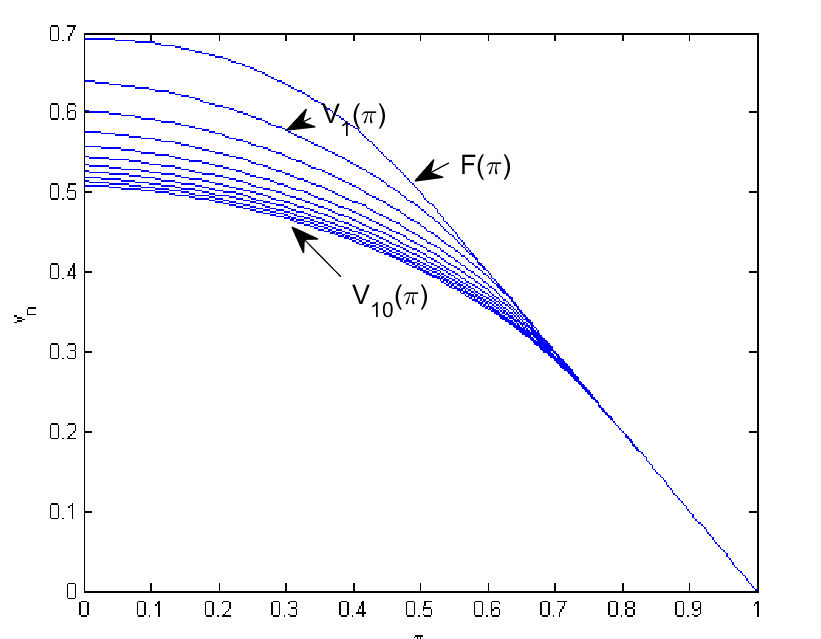}
\end{center}
\caption{$c=0.1,\lambda=0.1,\alpha=1,d=0.001,n=0,1,\cdots,10$.}\label{42new}
\end{figure}
\begin{figure}[H]
\begin{center}
\includegraphics[scale=.7]{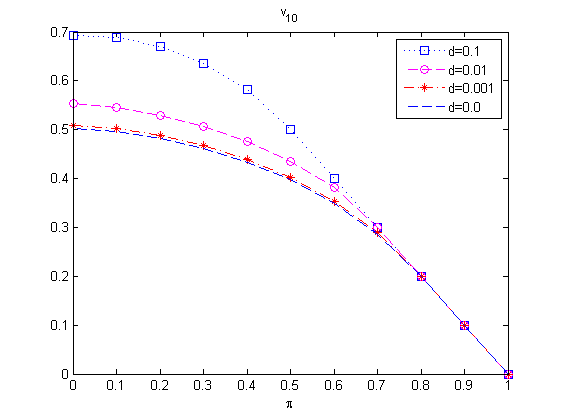}
\end{center}
\caption{$c=0.1,\lambda=0.1,\alpha=1$.}\label{43new}
\end{figure}

\begin{figure}[H]
\begin{center}
\includegraphics[scale=.7]{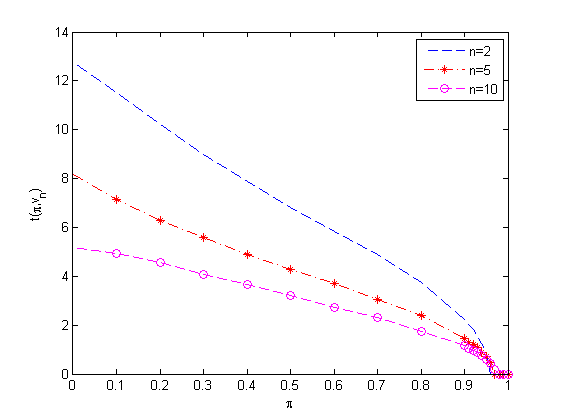}
\end{center}
\caption{$c=0.01,\lambda=0.1,\alpha=1,d=0.001$.}\label{44new.png}
\end{figure}
\begin{figure}[H]
\begin{center}
\includegraphics[scale=.4]{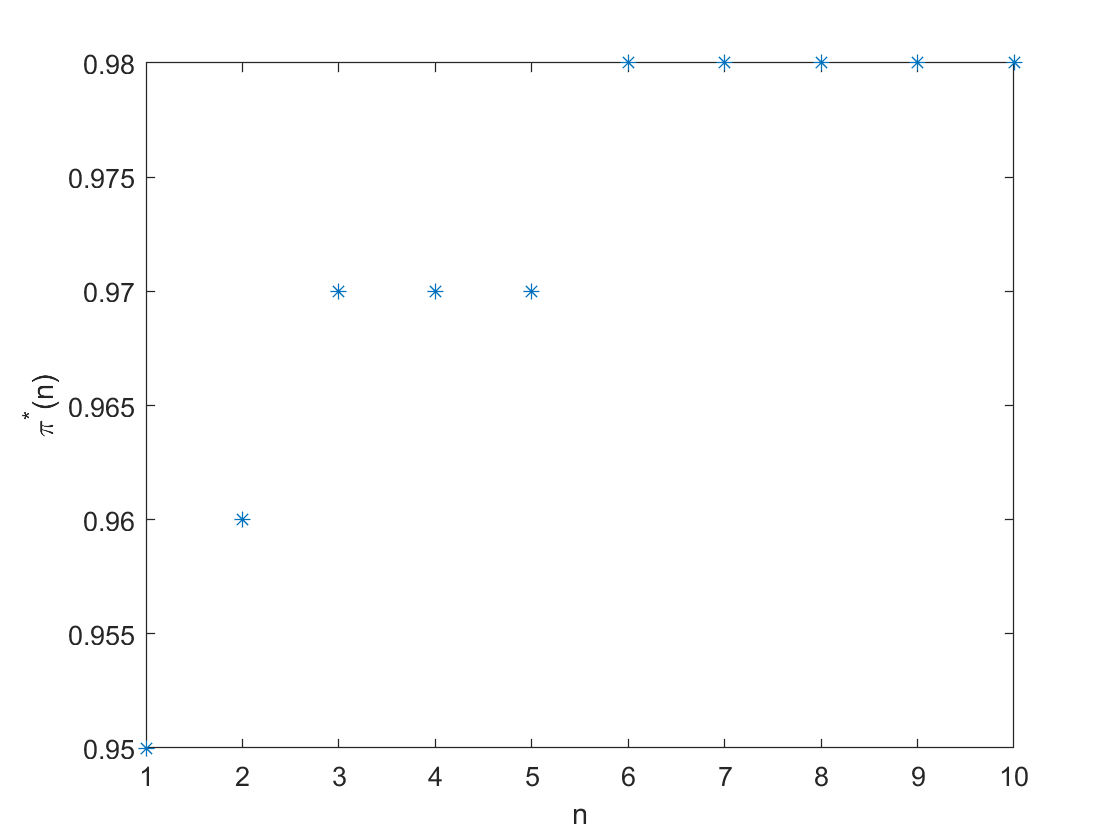}
\end{center}
\caption{$c=0.01,\lambda=0.1,\alpha=1,d=0.001$.}\label{44new1.png}
\end{figure}

\bibliography{ref}{}
\bibliographystyle{siam}
\end{document}